\documentclass[11pt,oneside]{amsart}
\usepackage{graphicx,color}
\usepackage{amsmath}
\usepackage{amssymb}
\usepackage{epsfig}
\usepackage{graphicx,color}
\usepackage[latin1]{inputenc}
\usepackage{latexsym}
\usepackage[normalem]{ulem}
\usepackage{graphpap}
\usepackage{bm}
\usepackage[shortlabels]{enumitem}
\usepackage{mathrsfs}
\usepackage[pdftex]{hyperref}

\newtheorem{teo}{Theorem}[section]
\newtheorem{defi}{Definition}[section]
\newtheorem{obs}{Remark}[section]
\newtheorem{lema}{Lemma}[section]
\newtheorem{prop}{Proposition}[section]
\newtheorem{ex}{Example}[section]
\newtheorem{corollary}{Corollary}[section]

\newcommand{\Eo}{\mathscr{E}_{\Omega} }
\newcommand{\Fo}{\mathscr{F}_{\Omega} }
\newcommand{\Go}{\mathscr{G}_{\Omega} }
\newcommand{\Lo}{\mathscr{L}_{\Omega} }
\newcommand{\Mu}{\mathscr{M}_{1\Omega} }
\newcommand{\Md}{\mathscr{M}_{2\Omega} }
\newcommand{\No}{\mathscr{N}_{\Omega} }

\newcommand{\R}{\mathbb{R}}

\newcommand{\cal}{\mathcal}
\newcommand{\bI}{\mathbf{I}}
\newcommand{\bII}{\mathbf{II}}

\newcommand{\PP}{\mathbf{P}}

\title[Line Congruences on singular surfaces ]{
	Line Congruences on singular surfaces   }

\author{$\text{D. Lopes}^\dagger,\; \text{T.A. Medina-Tejeda}^\star,\;\text{M.A.S. Ruas}^\ddagger\; \text{and}\; \text{I.C. Santos}^*$}

\thanks{\\ ${^\dagger} \text{Partially supported by INCTMat/CAPES Proc. 88887.510549/2020-00.}\\
		{^\star}\text{Supported by FAPESP Proc. 2021/11253-3.}\\ 
	{^\ddagger}\text{Partially supported by FAPESP Proc. 2019/21181-0 and CNPq
		Proc. 305695/2019-3.}\\ 
	{^*}\text{Supported by CAPES Proc. PROEX-11365975/D}$}

\address{ 
	Débora Lopes da Silva,	Departamento de Matemática, Universidade Federal do Sergipe, Av. Marechal Rondon, s/n Jardim Rosa Elze - CEP 49100-000 São Cristóvao, SE, Brazil.}
\email{  debora@mat.ufs.br}

\address{Tito Alexandro Medina Tejeda, Instituto de Ciências Matemáticas e de Computacão - Universidade de São Paulo, Av. Trabalhador sao-carlense, 400 - Centro, CEP: 13566-590 - São Carlos - SP, Brazil.}
\email{ talexanmedinat@gmail.com}

\address{ Maria Aparecida Soares Ruas, Instituto de Ciências Matemáticas e de Computacão - Universidade de São Paulo, Av. Trabalhador sao-carlense, 400 - Centro, CEP: 13566-590 - São Carlos - SP, Brazil.}
\email{ maasruas@icmc.usp.br}

\address{
	Igor Chagas Santos, Instituto de Ciências Matemáticas e de Computacão - Universidade de São Paulo, Av. Trabalhador sao-carlense, 400 - Centro, CEP: 13566-590 - São Carlos - SP, Brazil.}
\email{ igor.chs34@gmail.com}

\begin{document}
	
	\begin{abstract}
		This paper is a first step in order to extend Kummer's theory for line congruences to the case $\lbrace x, \xi \rbrace
		$, where $x: U \to \R^3$ is a smooth map and $\xi: U \to \R^3$ is a proper frontal. We show that if $\lbrace x, \xi \rbrace$ is a normal congruence, the equation of the principal surfaces is a multiple of the equation of the developable surfaces, furthermore, the multiplicative factor is associated to the singular set of $\xi$.

		\end{abstract}
	
	\maketitle
	
	\section{ INTRODUCTION}\label{sec:int} 
	
	A line congruence in the Euclidean space of dimension $3$ is a   $2-$parameter family of lines in $\mathbb R^{3}$. 
	
The first record about line congruences appeared in  ``M\'emoire sur la Th\'eorie des D\'eblais et des Remblais" (1776,1784) where Gaspard Monge seeks to solve a minimizing cost problem of transporting an amount of land from one place to another, preserving the volume (see \cite{ghys} for historical notes).  After Monge, Ernst Eduard Kummer in ``Allgemeine Theorie der geradlinigen Strahien systeme" (see \cite{Eisenhart}) was the first to deal with the general theory of line congruences. This theory is currently known as {\it Kummer theory of line congruences}. In recent years, the subject achieved an important development with contributions by
\cite{craizer2020curvature}, \cite{izumiya2}, \cite{Izumiya}, \cite{nossoartigo},  \cite{samuel1} among others.



 A line congruence is defined by a pair $\{x, \xi\} \in C^{\infty}(U, \R^3 \times \R^3 \setminus \lbrace \bm{0} \rbrace)$ where $U \subset \R^2$ is an open set and for all $(u_1,u_{2}) \in U$,  $ \xi (u_1,u_{2})$ denotes the direction of the line through the point $ x (u_1,u_{2}) $. An example of line congruences is given by the normal lines of a surface in $\R^3$. Each smooth curve $C$  in  $ x(U)$  is the directrix curve of a ruled surface whose generators are the lines of the congruence passing through points of $C.$  These ruled surfaces are called {\it surfaces of congruences}. Line congruences and ruled surfaces are often used in optics, mechanics, space kinematics and robot motion planning (\cite{marle2015william}, \cite{selig2013geometrical}).

 Izumiya, Saji and Takeuchi, in \cite{Izumiya}, represent (locally) a line congruence
	as the image of the map $F_{(x,\xi)}:U\times I\rightarrow\R^3$ where $F_{(x,\xi)}(u_1,u_2,t)=x(u_1,u_2)+t\xi(u_1,u_2)$, $x:U\rightarrow\R^3, \xi:U\rightarrow\R^3-\{0\}$ are smooth mappings, $U\subset\R^2$ is an open region and $I$ is an open interval. In this sense, we can see that line congruences generalize the concept of parallel surfaces. 
	
Most  known results in Kummer's theory are formulated  for congruences
$\{ x, \xi \}$ where $x$ is a regular surface and $\xi$ is an immersion.  For instance, we discuss  in  Proposition \ref{Cong_sup_desen}  a nice way of defining 
lines of curvature using line congruences:  lines of curvature on a smooth surface are those curves
whose surfaces of congruence $S_C$ are developable.

 Our goal in this paper is to extend this theory to the case of line congruences
$\{x, \xi\}$ where $x $ is a smooth map and $\xi$  is a proper frontal. For these line congruences,
we define the first and second Kummer fundamental forms associated to a moving basis $\Omega$ of the frontal $\xi.$
Based on results of \cite{Medina}, \cite{Medina2} and \cite{Medina3}, we define the $\Omega$-Kummer curvature function, denoted by $\mathscr{K}_{q}^{\Omega}$.
Similarly to  the case in which $x$ is smooth, for each point $p \in x(U),$ the directions where this curvature assumes extreme values are the {\it Kummer principal directions.}
We determine the equation of the developable surfaces and the equation of the principal surfaces of the congruence. Our main result is Theorem \ref{teoprincipal}, in which we prove that for normal line congruences, the equation of the principal surfaces is a multiple of the equation of the developable surfaces, furthermore, the multiplicative factor is associated to the singular set of $\xi$. As a corollary, we obtain the
following extension of Proposition \ref{Cong_sup_desen}: If $x$ and $\xi$ are proper frontals with the same singular sets and such that  $\xi$ is the unit normal vector of $x$, then a curve on $x$ is a directrix curve of a principal surface of the congruence if and only if it is a line of curvature of $x.$

The paper is organized as follows. In section \ref{sec:back2} some basic concepts about line congruences and Kummer's theory are introduced. In section \ref{sec:back3} notation and some useful results about frontals are presented, taking into account the approach used in \cite{Medina}. Section \ref{sec:back4} is addressed to the study of line congruences $\lbrace x, \xi \rbrace$, where $x: U \to \R^3$ is a smooth map and $\xi: U \to \R^3$ is a proper frontal.


\section{Background about line congruences}\label{sec:back2} 
	
We now present basic concepts and properties of line congruences in $\R^3$. The classical theory in $\R^3$ has been given in \cite{Bianchi}, \cite{Eisenhart} and \cite{weatherburn}.

\subsection{Kummer fundamental forms}
In what follows, the space $\R^3$ is oriented by an once for all fixed orientation and it is endowed with the Euclidean inner product $\langle,\rangle$.

Let  $\cal C$  $= \lbrace x,\xi \rbrace $ be a line congruence in $\R^3$, where $\xi,x:U\subset\R^{2}\rightarrow\R^3$ are smooth functions, $\lVert \xi \rVert=1$ and $\xi$ is an immersion. The image $S=x(U)$ is called a \textit{reference set} of the congruence. If $x$ parametrizes a regular surface in $\R^3$ then $S$ is called \textit{reference surface}. Let $\alpha:I\subset\R \rightarrow U$ be a regular curve, given by $\alpha(t) = (u_1(t),u_{2}(t))$. Denote by $x(t) = x (\alpha(t))$, $ \xi (t) = \xi(\alpha(t)) $, $ q = (u_1(0),u_{2}(0)) $, $ v = u_1'(0) x_{u_1 } (q) +u_{2}'(0) x_{u_{2} } (q) =
(u_1'(0),u_2'(0))\in T_{p} S $, where $p = x(q)$. Since $\xi$ is an immersion we could also consider $ w = u_1'(0) \xi_{u_1 } (q) +u_{2}'(0) \xi_{u_{2} } (q)$ in what follows, that is, we could take the tangent space of $\xi$ using the same coordinates $(u_1'(0),u_2'(0))$.

The following quadratic forms are associated to the congruence $ \cal C $.

\begin{enumerate}
	\item[(I)] Kummer first fundamental form :
	\begin{align}
	\mathcal{I}_{p}: T_{p}S &\rightarrow \R \\ 
	v &\mapsto \mathcal{I}_{p}(v) = \mathscr{E}u_1'^2 + 2\mathscr{F}u_1'u_2' + \mathscr{G}u_2'^2,\nonumber
	\end{align}
	where $\mathscr{E} = \langle \xi_{u_1}, \xi_{u_1} \rangle$, $\mathscr{F} = \langle \xi_{u_1}, \xi_{u_2} \rangle$ and $\mathscr{G} = \langle \xi_{u_2}, \xi_{u_2} \rangle$. We denote by $\bm{\mathcal{I}}$ the associated matrix.  
	\item[(II)] Kummer second fundamental form:
	\begin{align}
	\mathcal{II}_{p}: T_{p}S &\rightarrow \R\\
	v &\mapsto \mathcal{II}_{p}(v) = \mathscr{L}u_1'^2 + \left(\mathscr{M}_{1}+\mathscr{M}_{2} \right) u_1'u_2' + \mathscr{N}u_2'^2, \nonumber
	\end{align}
	where $\mathscr{L} = -\langle x_{u_1}, \xi_{u_1}\rangle$, $\mathscr{M}_{2} = -\langle x_{u_1}, \xi_{u_2}\rangle$, $\mathscr{M}_{1} = -\langle x_{u_2}, \xi_{u_1}\rangle$ and $\mathscr{N} = -\langle x_{u_2}, \xi_{u_2}\rangle$ . We denote by $\bm{\mathcal{II}} = -D\xi^{T} Dx$ the matrix of these last coefficients, where $D\xi, Dx$ denote the Jacobian matrices of $\xi, x$ respectively.
\end{enumerate}

In order to have $\bm{\mathcal{I}}$  positive definite we suppose that $\xi$ is an immersion. Usually, the Kummer second fundamental form is defined with the opposite sign (see \cite{Eisenhart} or \cite{weatherburn}), but in this paper we work with the above definition. Despite this, there are no changes on the geometry of the congruences.
The quadratic forms defined above are called Kummer's quadratic forms of the congruence, because it was Kummer in Allgemeine Theorie der Gradlinigen Strahen system who gave the first purely mathematical treatment of line congruences, see \cite{Eisenhart}. 

\begin{defi}\normalfont\normalfont\label{congrSurface}
	The lines of the congruence which pass through a curve $C$ on the reference surface $S$ form a ruled surface  $S_C$ called \textit{surface of the congruence}. 
\end{defi}

If $C$ is given by $x (\alpha(t))$ where $\alpha(t) = (u_1(t),u_2(t))$ and $ \xi (t) = \xi(\alpha(t)) $,  the surface of the congruence $S_C$ can be written as
\begin{equation}\label{ruled}
Y(t,w) = x(t) + w \xi(t), t \in I, w \in \R,
\end{equation}
where the curve $\alpha(t)$ is called a \textit{directrix} of $S_{C}$ and for each fixed $t$ the line $L_{t}$, which pass through $\alpha(t)$ and is parallel to $\xi(t)$, is called a \textit{generator} of the ruled surface $S_C$. If $||\xi(t)||=1$, we say that $\xi(t)$ is the \textit{spherical representation} of $S_{C}$.
Suppose $||\xi(t)||=1$ and $||\xi'(t)||\neq0$. It is known (see section 3.5 in \cite{docarmo}) that there exists a curve $\beta: I \rightarrow \R^3$, contained in the ruled surface $S_C$, parametrized by
\begin{align}\label{stric}
\beta(t) = x(t) + k(t)\xi(t),
\end{align}
where $k(t) = -\dfrac{\langle x'(t), \xi'(t) \rangle}{\langle \xi'(t), \xi'(t) \rangle}$, whose tangent vector satisfies 
\begin{equation}\label{prop_striction}
\langle \beta '(t), \xi'(t) \rangle = 0.
\end{equation}
This special curve is called \textit{striction line}.
 The intersection point of a generator with the striction line is called the  \textit{central point} of the generator. Given a generator $L_{t}$ the coordinate of its central point is $k(t)$, given in (\ref{stric}) . (See figure \ref{Sg})

		\begin{figure}[h!]
	\begin{center}
			\includegraphics[scale=0.4]{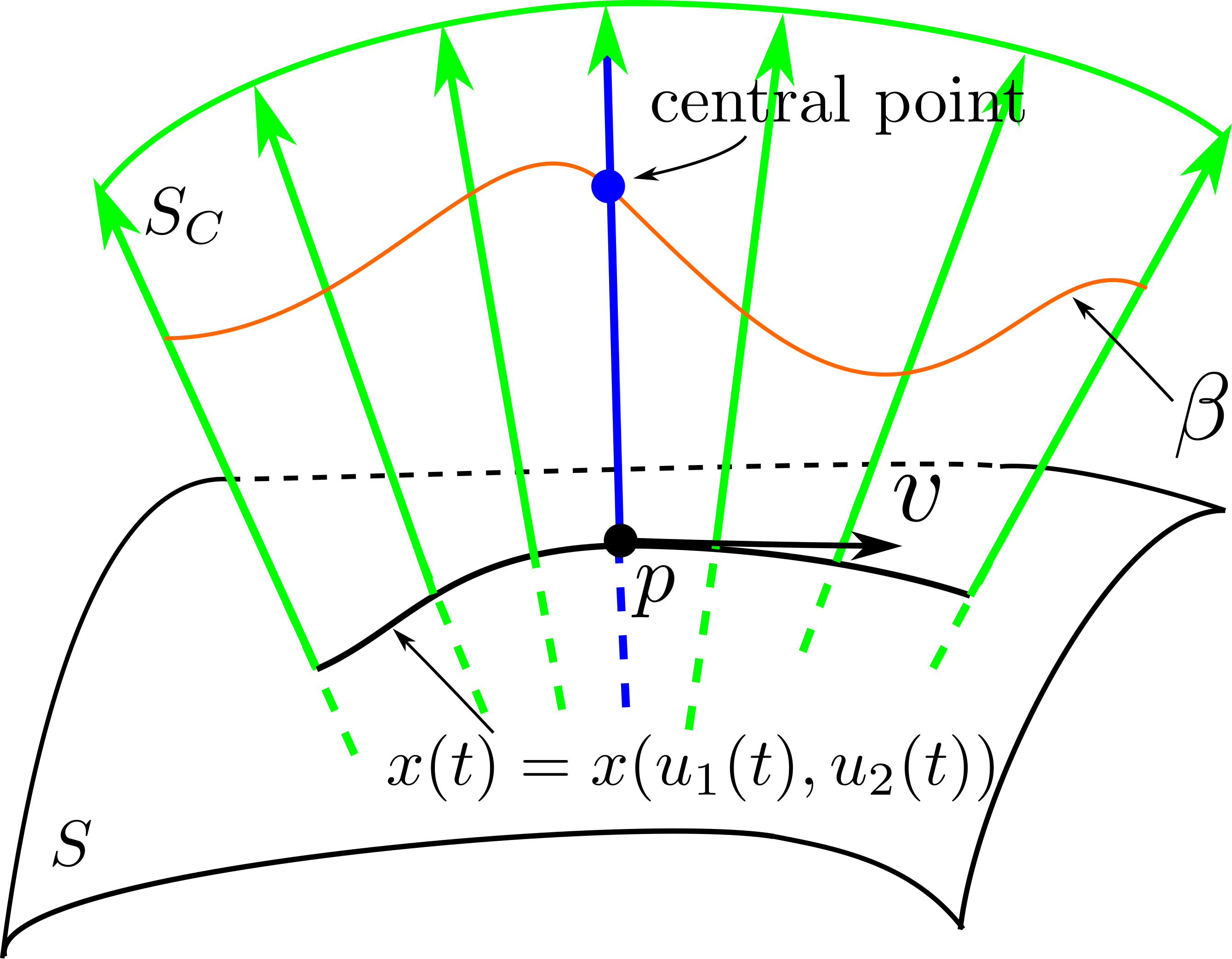}
			\caption{Surface of the congruence $S_C$, striction line $\beta$ and central point.}
		\label{Sg}

	\end{center}
\end{figure}

Let $q=(u_1(0),u_2(0))$ and note that 
\begin{align*}
k(0) &= -\dfrac{\langle x'(0), \xi'(0)\rangle }{\langle \xi'(0), \xi'(0) \rangle}\\
&= -\dfrac{\langle u_1'(0)x_{u_1}(q) + u_2'(0)x_{u_2}(q), u_1'(0)\xi_{u_1}(q) + u_2'(0)\xi_{u_2}(q) \rangle}{\langle u_1'(0)\xi_{u_1}(q) + u_2'(0)\xi_{u_2}(q), u_1'(0)\xi_{u_1}(q) + u_2'(0)\xi_{u_2}(q) \rangle}\\
&= \dfrac{\mathscr{L}u_1'^2 + \left(\mathscr{M}_{1}+\mathscr{M}_{2} \right) u_1'u_2' + \mathscr{N}u_2'^2}{\mathscr{E}u_1'^2 + 2\mathscr{F}u_1'u_2' + \mathscr{G}u_2'^2}\\
&=  \dfrac{{\mathcal{II}}_{p}}{{\mathcal{I}}_{p}},\; \text{where}\; p = x(q).
\end{align*}
If we associate to $v = \alpha'(0) = u_1'(0)x_{u_1}(q) + u_2'(0)x_{u_2}(q)$ its coordinates $(u_1'(0), u_2'(0))$, then it is possible to look at $k$ as a function defined in $T_{p}S$, i.e
\begin{equation}\label{functionk}
\mathscr{K}_{p}: T_pS\rightarrow\R,\quad \mathscr{K}_{p}(v) = \dfrac{{\mathcal{II}}_{p}(v)}{{\mathcal{I}}_{p}(v)},
\end{equation}
 which gives the coordinate of the central point of the generator $L_0$ associated to the surface of the congruence $S_C$, determined by $\alpha$, see figure \ref{Sg}. A point $p$ where the function $\mathscr{K}_{p}$ is constant is called a \textit{Kummer umbilic point}.

The directions $(du_1,du_2)$ where $\mathscr{K}_{p}$ assumes extreme values are called \textit{Kummer principal directions} . These directions are given by 
\begin{equation}\label{equaDifPrinc}
L_{\xi}du_1^2 + M_{\xi}du_1du_2 + N_{\xi}du_2^2=0,
\end{equation}
where $L_{\xi}=2\mathscr{F} \mathscr{L}-(\mathscr{M}_{1} + \mathscr{M}_{2}  )\mathscr{E} $, $M_{\xi}=2(\mathscr{G} \mathscr{L}-\mathscr{E} \mathscr{N})$ and $N_{\xi}=\mathscr{G}(\mathscr{M}_{1} + \mathscr{M}_{2})-2\mathscr{F}\mathscr{N}$ (see section 95 in \cite{weatherburn} for details).

If $p$ is not a $\xi$-umbilic point then there exist two \textit{Kummer principal directions} $l_1$ and $l_2$. The correspondent values of $\mathscr{K}_{p}$ in these directions, $\mathscr{K}_1= \mathscr{K}_{p}(l_1)$ and $\mathscr{K}_2=\mathscr{K}_{p}(l_2)$, are called \textit{Kummer principal curvatures}. The integral curves of (\ref{equaDifPrinc}) are called \textit{Kummer principal lines}. The surfaces of the congruence determined by the Kummer principal lines are called \textit{principal surfaces}.

Taking $v=l_i$ ($i=1,2$) in (\ref{functionk}), we see that the values $\mathscr{K}_1$ and $\mathscr{K}_2$ are  solutions of $ \nabla \mathcal{II} +k \nabla \mathcal{I}=0$, that is, they are roots of the following equation:
\begin{align}\label{eq:coord:limites}
&(\mathscr{E}\mathscr{F}-\mathscr{G}^2)k^2 + [-\mathscr{L}\mathscr{G}-\mathscr{E}\mathscr{N}+\mathscr{F}(\mathscr{M}_{1}+\mathscr{M}_{2}) ]k\\
 &+ \mathscr{L}\mathscr{N}- \left(\dfrac{\mathscr{M}_{1}+\mathscr{M}_{2}}{2}\right)^2 = 0.\nonumber 
\end{align}

\begin{defi}\normalfont\normalfont	A line congruence $\mathcal{C}=\{x,\xi\}$ is called a \textit{normal congruence} if there exists a surface $S'$ such that the lines of the congruence are parallel to its normal lines. We say that $\mathcal{C}$ is an \textit{exact normal congruence} if $x$ is a regular surface and $\xi$ is a normal vector to $x$.
\end{defi}

The notion of normal congruence plays an important role in classical  geometry of surfaces. It arose in an attempt to solve a problem of optimizing the cost of transport between two objects (surface) in the space, see \cite{Monge}. Proposition \ref{teonormal1} gives, in terms of Kummer's second fundamental form, a necessary and sufficient condition for the line congruence to be normal.

\begin{prop}\label{teonormal1}
	Let $\cal C$  $= \lbrace x, \xi \rbrace $. The congruence $\cal C$ is normal if, and only if, $\mathscr{M}_{1}=\mathscr{M}_{2}$, where $\mathscr{M}_{1}$ and $\mathscr{M}_{2}$ are coefficients of Kummer second fundamental form. 
\end{prop}
\begin{proof}
 See  proposition 5.1 in \cite{Izumiya}.
\end{proof}

It is known that if $\{x,\xi\}$ is an exact normal congruence then the Kummer principal lines of the congruence coincide with the lines of curvature of $S$ and the principal surfaces are developable. In fact, we have the next proposition.

\begin{prop}\label{Cong_sup_desen}
Let $\lbrace x,\xi \rbrace$ be an exact normal line congruence. A curve $C$ on the reference surface parametrized by $x (\alpha): I \rightarrow \R^3$, where $\alpha(t) = (u_{1}(t), u_{2}(t))$ is a smooth curve in $U$, is a  line of curvature if and only if the surface of the congruence $Y(t,w)=\alpha(t)+w\xi(t)$ is developable.  
\end{prop}
\begin{proof}
Let $S_C$ be the surface of the congruence parametrized by $Y (t, w) = x(t) + w\xi(t)$. It is known that the ruled surface $S_C$ is developable if and only if $ \left[x', \xi', \xi \right]=0$.
We know that $\lVert \xi \rVert = 1$ and $\langle x', \xi \rangle = 0$, thus $\left[x', \xi', \xi \right]=0$ if and only if $\xi'(t) = k(t)x'(t)$ and from Rodrigues' curvature formula (see section 3.2 in \cite{docarmo}), $\alpha$ is a line of curvature.
\end{proof}

\begin{teo} \label{Sc:developable} There exist two developable surfaces of the congruence, not necessarily real, passing through each congruence line.
\end{teo}
\begin{proof}
 See section 97 in \cite{weatherburn}.
\end{proof}

A central point is called a \textit{focal point} when the surface of the congruence is developable. 

Let $Y (t, w) = x (\alpha_i(t)) + w\xi(\alpha_i(t))$ $(i=1,2)$ be the developable surfaces of the congruence given by theorem \ref{Sc:developable}. It is known that the coordinate $\rho$ of a focal point  satisfies the following quadratic equation (see section 97 in \cite{weatherburn})
\begin{equation}\label{eq:focal_coord}
    (-\mathscr{F}^2+\mathscr{G}\mathscr{E})\rho^2+(\mathscr{F}\mathscr{M}_{2}-\mathscr{N}\mathscr{E}+\mathscr{F}\mathscr{M}_{1}-\mathscr{G}\mathscr{L})\rho-\mathscr{M}_{2}\mathscr{M}_{1}+\mathscr{N}\mathscr{L}=0.
\end{equation}

 It is also known that the Kummer principal curvatures, $\mathscr{K}_1$ and $\mathscr{K}_2$, and the coordinates $\rho_1$ and $\rho_2$ of the focal points  satisfy the following system
\begin{align*}
\begin{cases}
\rho_{1} + \rho_{2} = \mathscr{K}_1+\mathscr{K}_2\\ 
\left(\mathscr{K}_1-\mathscr{K}_2\right)^2 - \left(\rho_{1}-\rho_{2}\right)^2 = \dfrac{\left(\mathscr{M}_{1}-\mathscr{M}_{2}\right)^2}{{\mathscr{E}}\mathscr{G}-\mathscr{F}^2}
\end{cases}.
\end{align*}

From the above equations one can see that the focal points $P_1$, $P_2$ and the limit points $P_{\xi}^1$, $P_{\xi}^2$ have the same midpoint and they are equal if and only if the congruence is normal. See figure \ref{Congruence surfaces}.

\begin{figure}[ht!]
	\begin{center}
		\def\svgwidth{0.6\textwidth}
		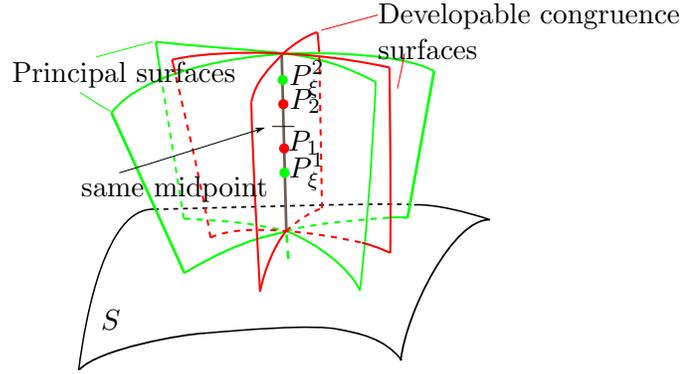
		\caption{Principal surfaces and developable surfaces.}
		\label{Congruence surfaces}
	\end{center}
\end{figure}

\section{Preliminaries on Frontals}\label{sec:back3} 
A smooth map $x: U \rightarrow \R^3$ defined in an open set $U\subset \R^2$ is called a \textit{frontal} if for all $q \in U$, there is a unit normal vector field $\nu: V_{q} \rightarrow \R^3$ along $x$, where $V_{q} \subset U$ is open and $q \in V_{q}$. This means, $\lVert \nu \rVert = 1$ and that $\nu$ is orthogonal to the partial derivatives of $x$ for each point $q \in V_{q}$. When the singular set $\Sigma(x) = \lbrace u \in U: x\; \text{is not immersive at u}\rbrace$ has empty interior, we call $x$ a \textit{proper frontal}. Since $\Sigma(x)$ is closed, this is equivalent to have $U \setminus\Sigma(x)$ being dense and open in $U$.

\begin{defi}\normalfont\normalfont
	We call \textit{moving basis} a smooth map $\Omega: U \rightarrow \cal{M}_{3\times 2}(\R)$ in which the columns $w_{1}, w_{2}: U \rightarrow \R^3$ of the matrix $\Omega =  \begin{pmatrix}
	w_{1} & w_{2}
	\end{pmatrix} $ are linearly independent vector fields.
\end{defi}

\begin{defi}\normalfont\normalfont
	We call a \textit{tangent moving basis} of $x$ a moving basis $\Omega = \left( w_{1}, w_{2} \right)$, such that $x_{u_{1}}, x_{u_{2}} \in \left[ w_{1}, w_{2} \right]_{\R}$, where $\left[\;, \right]_{\R}$ denotes the linear span vector space.
\end{defi}

It is known that a smooth map $x: U \to \R^3$ is a frontal if and only if there exist tangent moving bases of $x$ locally. Since we are interested
in local properties of frontals, we assume that we have a global
tangent moving basis $\Omega$ for $x$. Then, given a frontal, if we look at its Jacobian matrix as a smooth map $Dx: U \to M_{3 \times 2}(\R)$, we can always decompose $Dx = \Omega \Lambda_{\Omega}^{T}$, where $\Omega$ is a tangent moving basis and $\Lambda_{\Omega}: U \to M_{2 \times 2}(\R)$ is such that $\Sigma(x) = \lambda_{\Omega}^{-1}(0)$, where $\lambda_{\Omega} = \det \Lambda_{\Omega}$ (for details, see section 3 in \cite{Medina}).

Let $x: U \rightarrow \R^3$ be a frontal, $\Omega = \begin{pmatrix}
w_{1} & w_{2}
\end{pmatrix}$ a tangent moving basis and denote by $n = \dfrac{w_{1}\times w_{2}}{\lVert w_{1} \times w_{2} \rVert}$ the unit normal vector field induced by $\Omega$, which is also a frontal. We set the matrices
\begin{align*}
\bI_{\Omega} &:= \Omega^{T}\Omega, \\
\bII_{\Omega} &:= -\Omega^{T}Dn,\\
\mu_{\Omega} &:= -\bII_{\Omega}^{T}\bI_{\Omega}^{-1},\\
\alpha_{\Omega} &:= \mu_{\Omega}adj(\Lambda_{\Omega}).
\end{align*}
\begin{obs}\normalfont\normalfont
	With notation as above, if $\Omega$ is a tangent moving base of x, we have the decomposition $Dx = \Omega \Lambda_{\Omega}^T$, then $\Lambda_{\Omega} =Dx^{T}\Omega\mathbf{I}_{\Omega}^{-1}$, namely $\Lambda_{\Omega}$ is completely determined by $x$ and $\Omega$, therefore from now on, it will denote this matrix valued map. Also we write $P_{\Omega} = \left[ w_1, w_2 \right]_{\R}$ the plane generated by $w_1$ and $w_2$. Note that given two tangent moving basis $\Omega$ and $\tilde{\Omega}$ of a proper frontal, we have $P_{\Omega} = P_{\tilde{\Omega}}$. \end{obs}

In this paper, sometimes we deal with a pair of frontals $x$ and $\xi$, therefore in order to distinguish the matrix valued maps that have the same role of $\Lambda_{\Omega}$, we denote $\Delta_{\bar{\Omega}}$ the matrix such that $D\xi=\bar{\Omega}\Delta^T_{\bar{\Omega}}$ and $\delta_{\bar{\Omega}}:=det(\Delta_{\bar{\Omega}})$ where $\bar{\Omega}$ is a tangent moving basis of $\xi$. In particular, if we take $\xi=n$ the normal vector field induced by a tangent moving basis $\Omega$ of $x$, then is satisfied that $Dn=\Omega\mu_{\Omega}^T$ (see \cite{Medina}), that is, $\Delta_{\Omega}=\mu_{\Omega}$.

\begin{defi}\normalfont
	Let $x: U \rightarrow \R^3$ be a frontal and $\Omega$ a tangent moving basis of $x$, we define the $\Omega$-\textit{relative curvature} $K_{\Omega}:= \det(\mu_{\Omega})$ and the $\Omega$-\textit{relative mean curvature} $H_{\Omega}:= -\frac{1}{2}tr(\alpha_{\Omega})$, where $tr()$ is the trace and $adj()$ is the adjoint of a matrix. Also we call the functions $k_{1\Omega}:=H_{\Omega}-\sqrt{H_{\Omega}^2-\lambda_{\Omega} K_{\Omega}}$ and $k_{2\Omega}:=H_{\Omega}+\sqrt{H_{\Omega}^2-\lambda_{\Omega} K_{\Omega}}$ the $\Omega$-{\it relative principal curvatures}.  
\end{defi}
All these functions defined before play an important role in the differential geometry of frontals and are related with the classical Gaussian curvature, mean curvature and principal curvatures (see \cite{Medina,Medina2}). Eventually, these will appear at the end of the next section, where we consider the line congruence $\{x,\xi\}$ with $\xi=n$. 

\section{Line congruence on frontals}\label{sec:back4}
Most of the results in Kummer's theory are proved for congruences $\lbrace x, \xi \rbrace$, where $x: U \to \R^3$ is a regular surface and $\xi: U \to \R^3$ is an immersion. Our goal is to extend this theory to the case of line congruences $\lbrace x, \xi \rbrace$, where $x$ is a smooth map and $\xi$ is a proper frontal. We mentioned in section 2 that in the regular case it is possible to work using the tangent space of $x$ or the tangent space of $\xi$. Here, we replace these planes by the plane $P_{\Omega}$,  where $\Omega $ is a tangent moving basis of $\xi$.

\begin{ex}\normalfont
	Let $x: U \to \R^3$, where $U = (-1/10, 1/10) \times (-4, 4)$, defined by
$x = (u_{1}, u_{2}^2, {4/15}u_{1}u_{2}^5 + 1/2u_{1}^3u_{2}^4 + u_{1}u_{2}^2)$ (see figure \ref{frontalnp1}) and $\xi = \dfrac{1}{ \rho^{7/4}
}\left(\dfrac{-3\sqrt{3}}{8}\xi_1, \dfrac{9\sqrt{3}}{8}\xi_2, \dfrac{\sqrt{3}}{240}\xi_3 \right)$, where
\begin{align*}
\xi_1 &= 216\,{u_{1}}^{6}{u_{2}}^{4}-189\,{u_{1}}^{4}{u_{2}}^{5}+
66\,{u_{1}}^{2}{u_{2}}^{6}+16\,{u_{2}}^{7}+324\,{u_{1}}^{4}{u_{2}}^{2}\\
&+9\,{u_{1}}^{2}{u_{2}}^{3}+48\,{u_{2}}^{4}+108\,{u_{1}}^{2}+36\,u_{2}
\\
\xi_2 &= \left( 216\,{u_{1}}^{4}{u_{2}}^{4}+87\,{u_{1}}^{2}{u_{2}}^{5}-16\,{u_
	{2}}^{6}+252\,{u_{1}}^{2}{u_{2}}^{2}+24\,{u_{2}}^{3}+72 \right) {u_{2}
}^{2}\\
\xi_3 &= 145800\,{u_{1}}^{8}{u_{2}}^{8}+35721\,{u_{1}}^{6}{u_{2}}^{9}+25326\,{u
	_{1}}^{4}{u_{2}}^{10}+4896\,{u_{1}}^{2}{u_{2}}^{11} +6480\\ &+277020\,{u_{1}}^{6
}{u_{2}}^{6}+896\,{u_{2}}^{12}+114129\,{u_{1}}^{4}{u_{2}}^{7}+39204\,{
	u_{1}}^{2}{u_{2}}^{8}+5088\,{u_{2}}^{9}\\ &+179820\,{u_{1}}^{4}{u_{2}}^{4}
+88938\,{u_{1}}^{2}{u_{2}}^{5}+12096\,{u_{2}}^{6}+48600\,{u_{1}}^{2}{u
	_{2}}^{2}+14040\,{u_{2}}^{3}\\
\rho &= 54\,{u_{1}}^{4}{u_{2}}^{4}+9\,{u_{1}}^{2}{u_{2}}^{5}+4\,{u_{2}}^{6}+54
\,{u_{1}}^{2}{u_{2}}^{2}+12\,{u_{2}}^{3}+9.		 
\end{align*}
Then, $\lbrace x, \xi \rbrace$ is a special example of line congruence for which $x$ and $\xi$ are frontals and $\xi$ is an equiaffine transversal vector field different from the unit normal vector field of $x$. By equiaffine, we mean a vector field $\xi$ such that its derivatives $\xi_{u_1}$ and $\xi_{u_2}$ belong to $P_{\Omega}$, for all tangent moving basis $\Omega$. This is an important class of line congruences which will be further discussed in future works. For more information about equiaffine geometry see \cite{nomizu1994affine}.
\begin{figure}[h!]
	\includegraphics[scale=0.27]{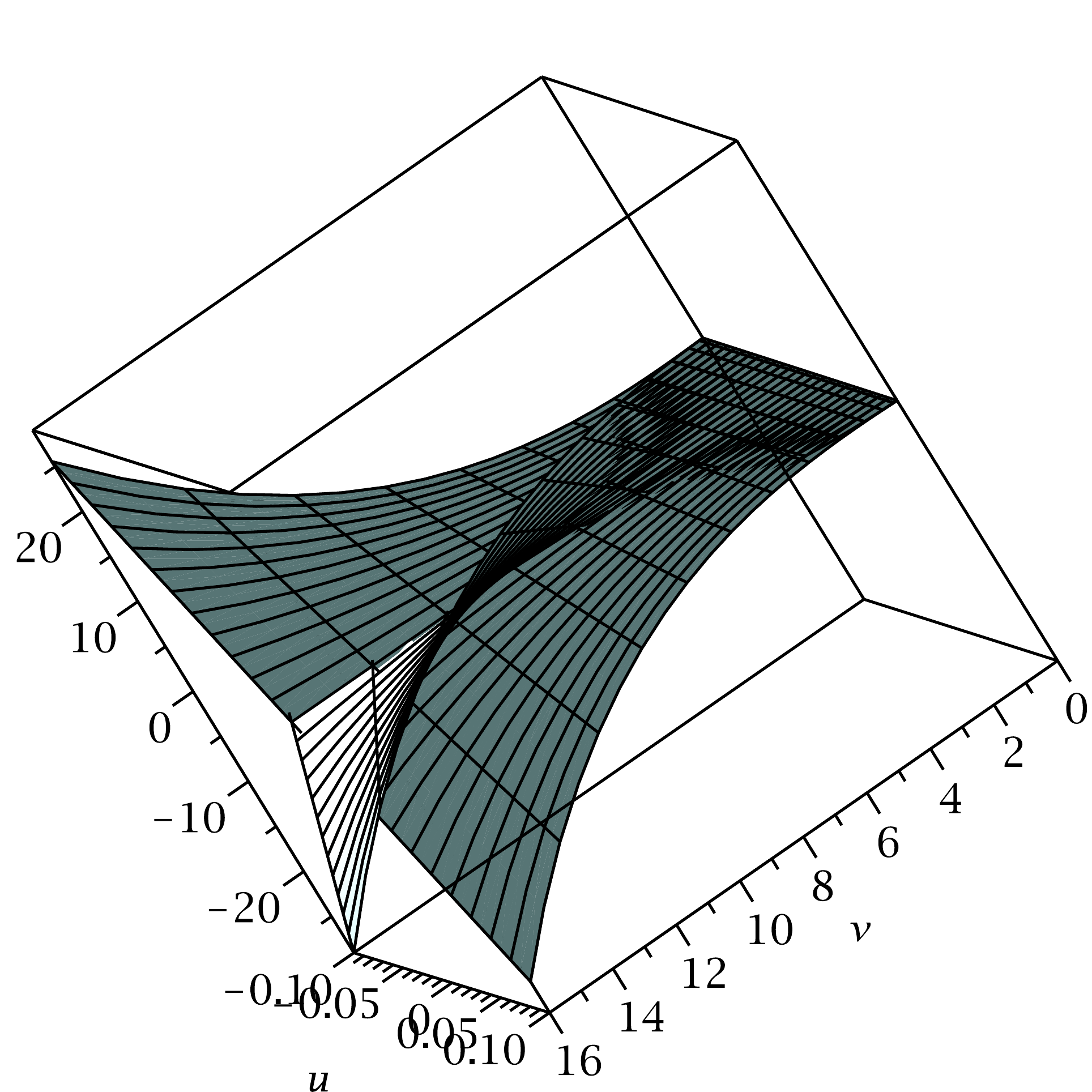}
	\includegraphics[scale=0.27]{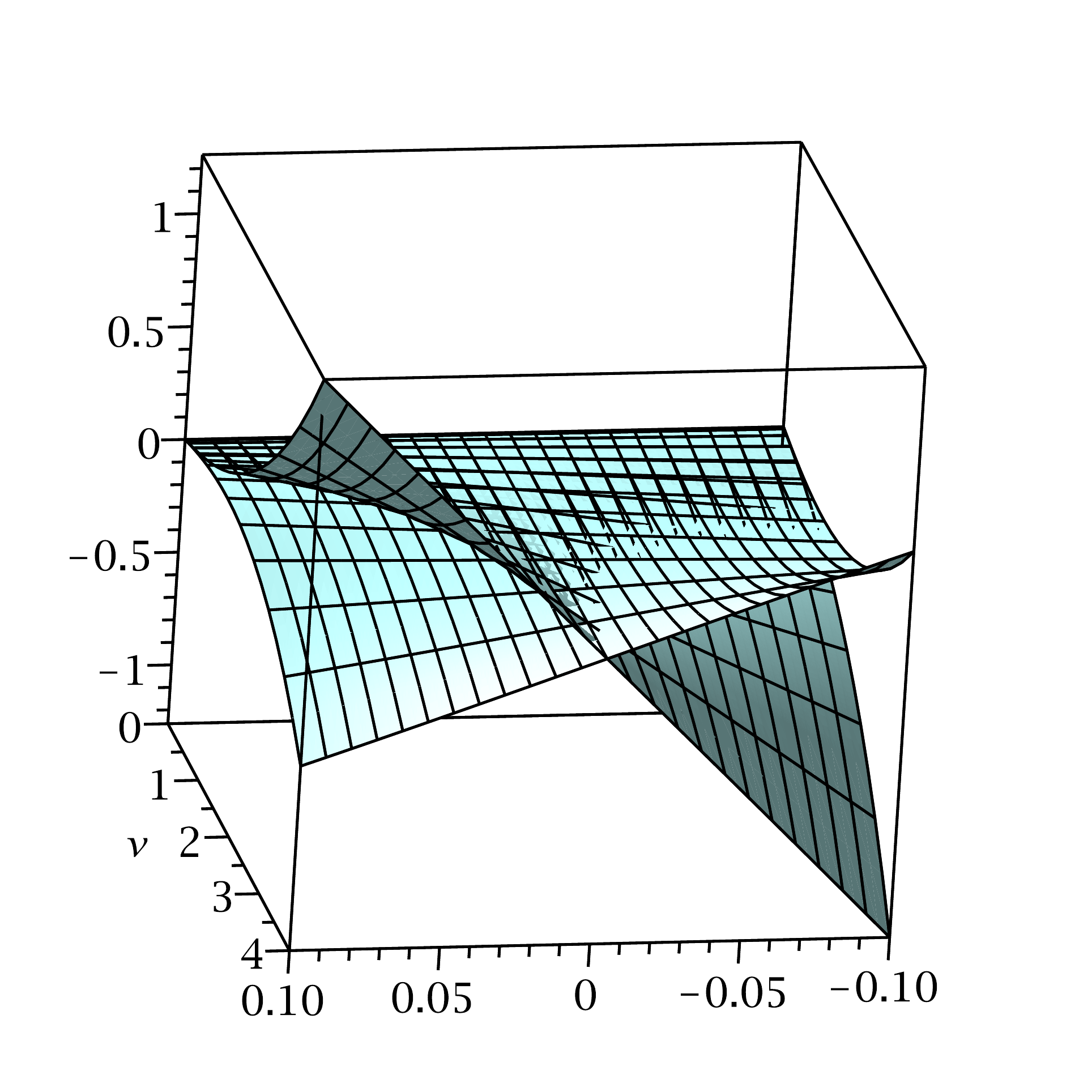}
	\caption{Frontal which admits an equiaffine transversal vector field different from its unit normal vector field.}
	\label{frontalnp1}
\end{figure}
\end{ex}

\begin{obs}\normalfont
	Along this paper, we consider several times that $\xi$ is a unitary frontal. In terms of family of lines, there is no difference considering $\xi$ unitary or not,  but we work with this restriction in order to  extend some concepts of Kummer's theory to this context. Next, we define the relative quadratic forms associated to a line congruence. \end{obs}

\begin{defi}\normalfont
Let $\cal{C} = \lbrace x, \xi \rbrace$ be a line congruence, where $\xi$ is a frontal and let $\Omega$ be a tangent moving basis of $\xi$. If $\Omega = \begin{pmatrix}
w_{1} & w_{2}
\end{pmatrix}$, we define the following quadratic forms: 
\begin{equation}
{\mathcal{I}}_{\Omega}(v) = \mathscr{E}_{\Omega} b_1^2 + 2\mathscr{F}_{\Omega}b_1b_2 + \mathscr{G}_{\Omega}b_2^2,
\end{equation}
where $\mathscr{E}_{\Omega} = \langle w_1, w_1 \rangle$, $\mathscr{F}_{\Omega} = \langle w_1, w_2 \rangle$, $\mathscr{G}_{\Omega} = \langle w_2, w_2 \rangle$, $v\in P_{\Omega}$ and $(b_1,b_2)$ are the coordinates of $v$ in the basis $w_1, w_2$. This form is called the \textit{ $\Omega$-Kummer first fundamental form} of $\cal{C}$ and we denote by $\bm{\mathcal{I}}_{\Omega}:=\Omega^T\Omega$ its associated matrix.
\begin{equation}
{\mathcal{II}}_{\Omega}(v) = \mathscr{L}_{\Omega}b_1^2 + \left(\mathscr{M}_{1\Omega}+\mathscr{M}_{2\Omega} \right) b_1b_2 + \mathscr{N}_{\Omega}b_2^2,
\end{equation}
where $\mathscr{L}_{\Omega} = -\langle x_{u_1}, w_1 \rangle$, $\mathscr{M}_{2\Omega} = -\langle x_{u_1}, w_2 \rangle$, $\mathscr{M}_{1\Omega} = -\langle x_{u_2}, w_1 \rangle$ and $\mathscr{N}_{\Omega} = -\langle x_{u_2}, w_2 \rangle $. This form is called the \textit{ $\Omega$-Kummer second fundamental form} of $\cal{C}$ and we denote by $\bm{\mathcal{II}}_{\Omega}:= -\Omega^{T}Dx$ the matrix of these last coefficients.
\end{defi}

\begin{obs}\normalfont
Note that ${\mathcal{I}}_{\Omega}$ is a positive-definite quadratic form.
\end{obs}

	Let $\cal{C} = \lbrace x, \xi \rbrace$ be a line congruence, where $\xi: U \to S^{2}$ is a frontal and let $\Omega$ be a tangent moving basis of $\xi$. Define the function $\mathscr{K}_{q}^{\Omega}: \R^2 \rightarrow \R$, given by
\begin{align}\label{funcaok}
\mathscr{K}_{q}^{\Omega}(b_{1}, b_{2}) &=  \dfrac{b^{T}{\bm{\mathcal{II}}}_{\Omega}adj(\Delta_{\Omega}^{T})b}{b^{T}{\bm{\mathcal{I}}}_{\Omega}b},
\end{align}
where $q \in U$, $adj()$ denotes the adjoint of a matrix, $b^{T} = \begin{pmatrix} b_{1} & b_{2} \end{pmatrix}$. Note that we can associate $(b_{1}, b_{2})$ to the coordinates of a vector  $v = b_{1}w_{1}(q) + b_{2}w_{2}(q) $, then we write $\mathscr{K}_{q}^{\Omega}(b_{1}, b_{2}) = \mathscr{K}_{q}^{\Omega}(v)$.

  \begin{prop}\label{propdecomposicao}
  		Let $\cal{C} = \lbrace x, \xi \rbrace$ be a line congruence, where $\xi$ is a frontal and let $\Omega$ be a tangent moving basis of $\xi$. Then
  		\begin{enumerate}
  			\item $\bm{\mathcal{I}} = \Delta_{\Omega}\bm{\mathcal{I}}_{\Omega} \Delta_{\Omega}^{T}$
  			\item  $\bm{\mathcal{II}} = \Delta_{\Omega} \bm{\mathcal{II}}_{\Omega}$.
  		\end{enumerate}
  	
  \end{prop}
  
  \begin{proof}
  	 Note that 
  	\begin{align*}
  	\bm{\mathcal{I}} &= D\xi^{T}D\xi = (\Omega\Delta_{\Omega}^{T})^{T}(\Omega\Delta_{\Omega}^{T})
  	= \Delta_{\Omega} \Omega^{T}\Omega\Delta_{\Omega}^{T}
  	= \Delta_{\Omega} \bm{\mathcal{I}}_{\Omega}\Delta_{\Omega}^{T}.
  	\end{align*}
The case for $\bm{\mathcal{II}}$ follows analogously. 
  
  \end{proof}
  
  Given a line congruence $\cal{C} = \lbrace x, \xi \rbrace$, where $\xi: U \to S^{2}$ is a frontal and $\Omega$ is a tangent moving basis of $\xi$,  the next proposition shows how the function $\mathscr{K}_{q}$, from Kummer's theory (given in \ref{functionk}) and the function $\mathscr{K}_{q}^{\Omega}$, defined in (\ref{funcaok}), are related when $q \notin \Sigma(\xi)$.
   
\begin{prop}\label{relacaofuncoesk}
	Let $\cal{C} = \lbrace x, \xi \rbrace$ be a line congruence, where $\xi: U \to S^{2}$ is a frontal and let $\Omega$ be a tangent moving basis of $\xi$. Then, for each $q \notin \Sigma(\xi)$, $\delta_{\Omega} \mathscr{K}_{q}(a_1,a_2) = \mathscr{K}_{q}^{\Omega}(b_1,b_2)$, where we write $b^{T} = (b_{1}, b_{2})\in \R^2$, $a^T=(a_1, a_2) = b^T\Delta_{\Omega}^{-1}(q)$ and $\delta_{\Omega} = \det \Delta_{\Omega}(q)$.
	\end{prop}
\begin{proof}
	Since $q \notin \Sigma(\xi)$, let $w = D\xi_{q}a$, hence we can write $w = \Omega \Delta_{\Omega}^{T} a$, where $a^{T} = \begin{pmatrix} a_{1} & a_{2} \end{pmatrix}$. Thus $a = \Delta_{\Omega}^{-T}b$, where $b^{T} = \begin{pmatrix}
	b_{1} & b_{2}
	\end{pmatrix}$ are the coordinates of $w$ relative to $\Omega$. From proposition \ref{propdecomposicao},
	\begin{align*}
	\mathscr{K}_{q}(a_{1}, a_{2}) &= \dfrac{a^{T} \bm{\mathcal{II}}a}{a^{T}\bm{\mathcal{I}}a}
	=  \dfrac{b^{T}\Delta_{\Omega}^{-1} \bm{\mathcal{II}}\Delta_{\Omega}^{-T}b}{b^{T}\Delta_{\Omega}^{-1}\bm{\mathcal{I}}\Delta_{\Omega}^{-T}b}
	=  \dfrac{b^{T}\Delta_{\Omega}^{-1}\left( \Delta_{\Omega}\bm{\mathcal{II}}_{\Omega}  \right)\Delta_{\Omega}^{-T}b}{b^{T}\Delta_{\Omega}^{-1}\left( \Delta_{\Omega}\bm{\mathcal{I}}_{\Omega} \Delta_{\Omega}^{T}  \right)\Delta_{\Omega}^{-T}b}\\
	&=\dfrac{1}{\det(\Delta_{\Omega})}\dfrac{b^{T} \bm{\mathcal{II}}_{\Omega}  adj(\Delta_{\Omega}^{T})b}{b^{T}\Delta_{\Omega}^{-1}\left( \Delta_{\Omega}\bm{\mathcal{I}}_{\Omega} \Delta_{\Omega}^{T}  \right)\Delta_{\Omega}^{-T}b}	
	\end{align*}
	hence
	\begin{align}\label{relacaofuncaok}
\delta_{\Omega}	\mathscr{K}_{q}(a_{1}, a_{2}) &=  \dfrac{b^{T} \bm{\mathcal{II}}_{\Omega}  adj(\Delta_{\Omega}^{T})b}{b^{T}\bm{\mathcal{I}}_{\Omega} b} = \mathscr{K}_{q}^{\Omega}(b_{1}, b_{2})
	\end{align}
		\end{proof}
	
Note that $\mathscr{K}_{q}^{\Omega}(a_1, a_2) = \mathscr{K}_{q}^{\Omega}(\varphi(a_1, a_2))$, for all $\varphi \in \R \setminus 0$, thus we can consider $\mathscr{K}_{q}^{\Omega}: S^{1} \rightarrow \R$.

\begin{prop}\label{propcongnormal}
	Let ${\cal C} = \lbrace x, \xi \rbrace$ be a line congruence, where $\xi: U \to S^2$ is a frontal. The congruence $\cal C$ is normal if and only if the matrix $\bm{\cal{II}}_{\Omega}adj(\Delta_{\Omega}^{T})$ is symmetric.
\end{prop}
\begin{proof}
It follows from proposition 5.1 in \cite{Izumiya} that $\cal{C}$ is normal if and only if $\bm{\mathcal{II}} = \Delta_{\Omega} \bm{\mathcal{II}}_{\Omega}$ is symmetric. It can be shown by straightforward calculations that  this is equivalent to say that	$\bm{\cal{II}}_{\Omega}adj(\Delta_{\Omega}^{T})$ is symmetric.
\end{proof}	

\begin{defi}\normalfont
	Let $ \lbrace x, \xi \rbrace$ be a line congruence, where $\xi: U \rightarrow \R^3$ is a frontal and let $C$ be a curve on $x$ parametrized by $x(t) = x(\alpha(t))$, where $\alpha: I \rightarrow U$ is smooth and $\xi(t) = \xi(\alpha(t))$ is the restriction of $\xi$ to $C$. The ruled surface $S_{C}$, parametrized by
		\begin{equation}\label{supcong}
	Y(t,v)=x(t)+v\xi(t),\quad t \in I\subset \R, v \in \R,
	\end{equation}
 is called a surface of the congruence.
\end{defi}

\begin{defi}\label{defdirprincipal}\normalfont
	Let $ \lbrace x, \xi \rbrace$ be a line congruence, where $\xi: U \rightarrow S^{2}$ is a proper frontal, $\Omega = \begin{pmatrix}
	w_{1} & w_{2}
	\end{pmatrix}$ a tangent moving basis of $\xi$. We say that a direction $w \in P_{\Omega}$ is a \textit{Kummer principal direction} if $\mathscr{K}_{q}^{\Omega}(w)$ is an extreme value of $\mathscr{K}_{q}^{\Omega}$.	
\end{defi}

\begin{obs}\normalfont\label{remarkindepenciabase}
	By analogy with the case of definition 3.2 in  \cite{Medina3}, one can get that the Kummer principal directions do not depend on the chosen tangent moving basis. 
\end{obs}

\begin{defi}\normalfont
	Let $ \lbrace x, \xi \rbrace$ be a line congruence, where $\xi: U \rightarrow S^{2}$ is a proper frontal and $\Omega$ a tangent moving basis of $\xi$. Let $S_{C}$ be a surface of the congruence, given by 
	\begin{equation}\label{supcong1}
	Y(t,v)=x(t)+v\xi(t),\quad t \in I\subset \R, v \in \R,
	\end{equation}
	where $\alpha: U \to I$, such that $\alpha(t) = (u_{1}(t), u_{2}(t))$ is smooth, $x(t) = x(\alpha(t))$ and $\xi(t) = \xi(\alpha(t))$.
	We say that $S_{C}$ is a \textit{principal surface} if for all $t \in I$ such that $\begin{pmatrix} b_{1} \\ b_{2} \end{pmatrix} = \Delta_{\Omega}^{T}\begin{pmatrix}
	u_{1}' \\ u_{2}'
	\end{pmatrix} \neq \mathbf{0}$, $(b_{1}, b_{2})$ determines a Kummer principal direction in $P_{\Omega}$. We call a directrix curve of a principal surface (or the associated curve on $U$, $\alpha: I \to U$) a \textit{Kummer principal line}.
\end{defi}

\begin{lema}
	Let $ \lbrace x, \xi \rbrace$ be a line congruence, where $\xi: U \rightarrow S^{2}$ is a proper frontal, $\Omega$ is a tangent moving basis of $\xi$ and we write $\Delta_{\Omega} = \begin{pmatrix}\delta_{ij} \end{pmatrix}$. Then, $\mathscr{K}_{q}^{\Omega}$ has an extreme value at $(b_{1}, b_{2})$ if and only if
	\begin{align}
	&b_{1}(\delta_{22}\mathscr{L}_{\Omega} - \delta_{12}\mathscr{M}_{1}) + \dfrac{b_{2}}{2}(\delta_{11}\mathscr{M}_{1\Omega} - \delta_{21}\mathscr{L}_{\Omega} + \delta_{22}\mathscr{M}_{2\Omega} - \delta_{12}\mathscr{N}_{\Omega}) \label{eq1}\\ 
	&- k_{0}(b_{1}\mathscr{E}_{\Omega} + b_{2}\mathscr{F}_{\Omega}) = 0\nonumber  \\
	&b_{2}(\delta_{11}\mathscr{N}_{\Omega} - \delta_{21}\mathscr{M}_{2\Omega}) + \dfrac{b_{1}}{2}(\delta_{11}\mathscr{M}_{1\Omega} - \delta_{21}\mathscr{L}_{\Omega} + \delta_{22}\mathscr{M}_{2\Omega} - \delta_{12}\mathscr{N}_{\Omega}) \label{eq2}\\ 
	&- k_{0}(b_{2}\mathscr{G}_{\Omega} + b_{1}\mathscr{F}_{\Omega}) = 0, \nonumber
	\end{align}
	where $k_{0} = \mathscr{K}^{\Omega}_{q}(b_{1}, b_{2})$.
\end{lema}

\begin{proof}
	Note that 
	\begin{align*}
	\mathscr{K}_{q}^{\Omega}(b_{1}, b_{2}) &= \dfrac{b_{1}^2(\delta_{22}\mathscr{L}_{\Omega} - \delta_{12}\mathscr{M}_{1}) + b_{2}^2 (\delta_{11}\mathscr{N}_{\Omega} - \delta_{21}\mathscr{M}_{2\Omega})}{b_{1}^2\mathscr{E}_{\Omega} + 2b_{1}b_{2}\mathscr{F}_{\Omega} + b_{2}^2\mathscr{G}_\Omega}\\
	&+ \dfrac{ b_{1}b_{2}(\delta_{11}\mathscr{M}_{1\Omega} - \delta_{21}\mathscr{L}_{\Omega} + \delta_{22}\mathscr{M}_{2\Omega} - \delta_{12}\mathscr{N}_{\Omega})}{b_{1}^2\mathscr{E}_{\Omega} + 2b_{1}b_{2}\mathscr{F}_{\Omega} + b_{2}^2\mathscr{G}_\Omega} 
	\end{align*}
	Let us suppose that $\mathscr{K}_{q}^{\Omega}$ has an extreme value at $(b_{1}, b_{2})$. At an extreme value $k_{0}$ of $\mathscr{K}^{\Omega}_{q}$, we have $\frac{\partial \mathscr{K}_{q}^{\Omega}}{\partial b_{i}} = 0$, $i=1,2$. From this, we get (\ref{eq1}) and (\ref{eq2}). Reciprocally, if $(b_{1}, b_{2})$ is such that  (\ref{eq1}) and (\ref{eq2}) are valid, then we have directly that $k_{0}$ is an extreme value of $\mathscr{K}^{\Omega}_{q}$. Let us show that $\mathscr{K}^{\Omega}_{q}(b_{1}, b_{2}) = k_{0}$. Let us suppose $b_{1} \neq 0$ and $b_{2} \neq 0$ (other cases are analogous). If we sum (\ref{eq1}) multiplied by $b_{1}$ with (\ref{eq2}) multiplied by $b_{2}$, we obtain
	\begin{align*}
	b_{1}^2 &(\delta_{22}\mathscr{L}_{\Omega} - \delta_{12}\mathscr{M}_{1}) + b_{1}b_{2}(\delta_{11}\mathscr{M}_{1\Omega} - \delta_{21}\mathscr{L}_{\Omega} + \delta_{22}\mathscr{M}_{2\Omega} - \delta_{12}\mathscr{N}_{\Omega}) \\&+ b_{2}^2 (\delta_{11}\mathscr{N}_{\Omega} - \delta_{21}\mathscr{M}_{2\Omega}) = k_{0}(b_{1}^2\mathscr{E}_{\Omega} + 2b_{1}b_{2}\mathscr{F}_{\Omega} + b_{2}^2\mathscr{G}_\Omega).
	\end{align*}
	From this, we have $k_{0} = \mathscr{K}^{\Omega}_{q}(b_{1}, b_{2})$.
\end{proof}

\begin{prop}\label{obsprincipal}
	Let $ \lbrace x, \xi \rbrace$ be a line congruence, where $\xi: U \rightarrow S^{2}$ is a proper frontal, $\Omega$ a tangent moving basis of $\xi$ and we write $\Delta_{\Omega} = \begin{pmatrix}\delta_{ij} \end{pmatrix}$. 
	\begin{enumerate}
		\item A curve $(u_{1}(t), u_{2}(t))$ is a Kummer principal line if and only if this is a solution of
		\begin{align}\label{eqsupprincipais}
		C_{1}b_{1}^{2} + C_{2}b_{1}b_{2} + C_{3}b_{2}^2 = 0,\; \text{for all $t$},
		\end{align} 
		where
		\begin{align*}
		C_{1} &= 2\mathscr{F}_{\Omega}\left(\delta_{22}\mathscr{L}_{\Omega} - \delta_{12}\mathscr{M}_{1} \right) - \mathscr{E}_{\Omega}\left(\delta_{11}\mathscr{M}_{1\Omega} - \delta_{21}\mathscr{L}_{\Omega} + \delta_{22}\mathscr{M}_{2\Omega} - \delta_{12}\mathscr{N}_{\Omega} \right)\\
		C_{2} &= 2\mathscr{G}_{\Omega} \left(\delta_{22}\mathscr{L}_{\Omega} - \delta_{12}\mathscr{M}_{1} \right) - 2\mathscr{E}_{\Omega} \left( \delta_{11}\mathscr{N}_{\Omega} - \delta_{21}\mathscr{M}_{2\Omega} \right)\\
		C_{3} &= \mathscr{G}_{\Omega}\left(\delta_{11}\mathscr{M}_{1\Omega} - \delta_{21}\mathscr{L}_{\Omega} + \delta_{22}\mathscr{M}_{2\Omega} - \delta_{12}\mathscr{N}_{\Omega} \right) - 2\mathscr{F}_{\Omega}\left( \delta_{11}\mathscr{N}_{\Omega} - \delta_{21}\mathscr{M}_{2\Omega} \right).
		\end{align*}
		We call (\ref{eqsupprincipais}) the \textit{equation of principal surfaces} of the congruence.
		\item If the congruence is normal then (\ref{eqsupprincipais}) can be written as
		\begin{align}\label{repmatricialpri}
		\begin{pmatrix}
		u_{1}' & u_{2}'
		\end{pmatrix}\Delta_{\Omega} \PP adj(\bm{\cal{II}}_{\Omega})^{T}\Delta_{\Omega}\bm{\mathcal{I}}_{\Omega}\Delta_{\Omega}^{T}\begin{pmatrix}
		u_{1}'\\ u_{2}'
		\end{pmatrix} = 0,
		\end{align}
		where $\PP = \begin{pmatrix}
		0 & 1\\
		-1 & 0
		\end{pmatrix}$.
	\end{enumerate}
\end{prop}

\begin{proof}
	\begin{enumerate}
		
		\item 
		From (\ref{eq1}) and (\ref{eq2}), we have that $(b_{1}(t), b_{2}(t))$ provides an extreme value of $\mathscr{K}^{\Omega}_{q}$ for all $t$  if and only if
		\begin{align*}
		\begin{vmatrix}
		b_{1}(\delta_{22}\mathscr{L}_{\Omega} - \delta_{12}\mathscr{M}_{1}) + \dfrac{b_{2}}{2}\mathcal{M} & b_{1}\mathscr{E}_{\Omega} + b_{2}\mathscr{F}_{\Omega}\\
		b_{2}(\delta_{11}\mathscr{N}_{\Omega} - \delta_{21}\mathscr{M}_{2\Omega}) + \dfrac{b_{1}}{2}\mathcal{M} &  b_{2}\mathscr{G}_{\Omega} + b_{1}\mathscr{F}_{\Omega}
		\end{vmatrix} = 0,
		\end{align*}
		where $\mathcal{M} = (\delta_{11}\mathscr{M}_{1\Omega} - \delta_{21}\mathscr{L}_{\Omega} + \delta_{22}\mathscr{M}_{2\Omega} - \delta_{12}\mathscr{N}_{\Omega})$. The equation (\ref{eqsupprincipais}) is obtained directly from the above expression.
		
		\item We know from proposition \ref{propcongnormal} that $\lbrace x, \xi \rbrace$ is normal if and only if $\bm{\cal{II}}_{\Omega}adj(\Delta_{\Omega}^{T})$ is symmetric, which is equivalent to say that
		\begin{align}\label{condnormal}
		\delta_{11}\mathscr{M}_{1\Omega} - \delta_{21}\mathscr{L}_{\Omega} = \delta_{22}\mathscr{M}_{2\Omega} - \delta_{12}\mathscr{N}_{\Omega}.
		\end{align}
		By using this condition in (\ref{eqsupprincipais}), we obtain (\ref{repmatricialpri}).

	\end{enumerate}
	
\end{proof}

\begin{prop}\label{propdisc}
	The discriminant $\mathcal{D} = C_{2}^2 - 4C_{1}C_{3}$ of the equation (\ref{eqsupprincipais}) is non-negative. This discriminant is zero if and only if the coefficients $C_{1}, C_{2}$ and $C_{3}$ are identically zero.
\end{prop}

\begin{proof}
	If we write $\mathcal{L} = 2\left(\delta_{22}\mathscr{L}_{\Omega} - \delta_{12}\mathscr{M}_{1} \right)$, $\mathcal{N} = 2\left( \delta_{11}\mathscr{N}_{\Omega} - \delta_{21}\mathscr{M}_{2\Omega} \right)$ and $\mathcal{M} = \left(\delta_{11}\mathscr{M}_{1\Omega} - \delta_{21}\mathscr{L}_{\Omega} + \delta_{22}\mathscr{M}_{2\Omega} - \delta_{12}\mathscr{N}_{\Omega} \right)$, then
	\begin{align*}
	C_{1} &= \mathscr{F}_{\Omega}\mathcal{L} - \mathscr{E}_{\Omega}\mathcal{M}\\
	C_{2} &= \mathscr{G}_{\Omega}\mathcal{L} - \mathscr{E}_{\Omega}\mathcal{N}\\
	C_{3} &= \mathscr{G}_{\Omega}\mathcal{M} - \mathscr{F}_{\Omega}\mathcal{N}.
	\end{align*}
	Hence, $\mathscr{G}_{\Omega}C_{1} - \mathscr{F}_{\Omega}C_{2} + \mathscr{E}_{\Omega}C_{3} = 0 $, from which we get $C_{3} = \dfrac{\mathscr{F}_{\Omega}C_{2} - \mathscr{G}_{\Omega}C_{1}}{\mathscr{E}_{\Omega}}$. Thus
	\begin{align*}
	\mathcal{D} &= C_{2}^2 - 4\dfrac{C_{1}\mathscr{F}_{\Omega}C_{2}}{\mathscr{E}_{\Omega}} + 4\dfrac{\mathscr{G}_{\Omega}C_{1}^2}{\mathscr{E}_{\Omega}}\\
	&= \left( C_{2} - 2 \dfrac{C_{1}\mathscr{F}_{\Omega}}{\mathscr{E}_{\Omega}} \right)^{2} - 4 \dfrac{C_{1}^2 \mathscr{F}_{\Omega}^{2}}{\mathscr{E}^{2}} + 4\dfrac{\mathscr{G}_{\Omega}C_{1}^2}{\mathscr{E}_{\Omega}}\\
	&= \left( C_{2} - 2 \dfrac{C_{1}\mathscr{F}_{\Omega}}{\mathscr{E}_{\Omega}} \right)^{2} + 4 \dfrac{C_{1}^2}{\mathscr{E}_{\Omega}^2}\left(\mathscr{E}_{\Omega}\mathscr{G}_{\Omega} - \mathscr{F}_{\Omega}  \right) \geq 0.
	\end{align*}
	As $\mathscr{E}_{\Omega}\mathscr{G}_{\Omega} - \mathscr{F}_{\Omega} > 0$, we get $\mathcal{D} = 0$ if and only if $C_{1} = C_{2} = 0$, which implies that $C_{3} = 0$.
\end{proof}	

Proposition \ref{propdisc} asserts that at points where $\mathcal{D}>0$ there are only two Kummer principal directions. 

We can also look at the developable surfaces associated to a given line congruence $\lbrace x, \xi \rbrace$, where $x: U \rightarrow \R^3$ is a smooth map and $\xi: U \rightarrow \R^3$ is a unitary proper frontal. In order to do this, we have the next proposition.

\begin{prop}\label{propdesenv}
	Let $ \lbrace x, \xi \rbrace$ be a line congruence, where $\xi: U \rightarrow S^2$ is a proper frontal, $\Omega$ a tangent moving basis of $\xi$. A surface of the congruence $Y(t,v) = x(u_{1}(t), u_{2}(t)) + v\xi(u_{1}(t), u_{2}(t))$ is a developable surface if and only if $(u_{1}(t), u_{2}(t))$ is a solution of
	\begin{align}\label{eqdesenvolsup}
	\begin{pmatrix}
	u_{1}' & u_{2}'
	\end{pmatrix} \PP adj(\bm{\cal{II}}_{\Omega})\bm{\mathcal{I}}_{\Omega}\Delta_{\Omega}^{T}\begin{pmatrix}
	u_{1}'\\ u_{2}'
	\end{pmatrix} = 0.
	\end{align} 
	We call (\ref{eqdesenvolsup}) the \textit{equation of developable surfaces} of the congruence.
\end{prop}

\begin{proof}
	Let us suppose $\alpha: I \rightarrow U$ a smooth curve, given by $\alpha(t) = (u_{1}(t), u_{2}(t))$, such that the surface of the congruence $Y(t,v) = x(t) + v\xi(t)$ is developable, where $x(t) = x (\alpha(t))$	and $\xi(t) = \xi(\alpha(t))$. Then it is known that $\left[x', \xi', \xi  \right] = 0$ (See section 3.5 in \cite{docarmo}). From this expression, we obtain the differential equation of developable surfaces
	\begin{align}\label{eqdesenv}
	u_{1}'^2\left[ x_{u_1}, \xi_{u_1}, \xi \right] + u_{1}'u_{2}' \left(\left[x_{u_1}, \xi_{u_2}, \xi   \right]  + \left[x_{u_2}, \xi_{u_1}, \xi   \right] \right) + u_{2}'^2 \left[x_{u_2}, \xi_{u_2}, \xi   \right] = 0.
	\end{align}
	By considering that $\xi$ is unitary we have $\xi = \dfrac{w_{1} \times w_{2}}{\lVert w_{1} \times w_{2} \rVert}$, where $\Omega = \begin{pmatrix} w_{1} & w_{2} \end{pmatrix}$.  We also know that $D\xi = \Omega\Delta_{\Omega}^{T}$, where $\Delta_{\Omega} = \begin{pmatrix}
	\delta_{ij}
	\end{pmatrix}$, thus 
	\begin{align*}
	\left[ x_{u_1}, \xi_{u_1}, \xi \right] = \frac{1}{\lVert w_{1} \times w_{2} \rVert}\langle x_{u_1}, (\delta_{11}w_{1} + \delta_{12}w_{2}) \times( w_{1} \times w_{2}) \rangle.
	\end{align*}
	By using the coefficients of the first and second $\Omega$-Kummer fundamental forms and the formula for the vector triple product ($\mathbf{a} \times (\mathbf{b} \times \mathbf{c}) = \langle \mathbf{a}, \mathbf{c} \rangle \mathbf{b} - \langle \mathbf{a}, \mathbf{b} \rangle \mathbf{c}$), we get
	\begin{align}\label{coefuu}
	\left[ x_{u_1}, \xi_{u_1}, \xi \right] = \delta_{11}(\mathscr{F}_{\Omega}\mathscr{L}_{\Omega} - \mathscr{E}_{\Omega}\mathscr{M}_{1\Omega}) + \delta_{12}\left( \mathscr{L}_{\Omega}\mathscr{G}_{\Omega} -\mathscr{F}_{\Omega} \mathscr{M}_{1\Omega} \right).
	\end{align}	
	In a similar way, we obtain
	\begin{align}
	\left[ x_{u_1}, \xi_{u_2}, \xi \right] + \left[ x_{u_2}, \xi_{u_1}, \xi \right] =& \delta_{21}(\mathscr{F}_{\Omega} \mathscr{L}_{\Omega}  - \mathscr{E}_{\Omega} \Mu) + \delta_{22} (\Go \Lo - \Fo \Mu) \label{coefuv}\\
	&+ \delta_{11}(\Fo \Md - \Eo \No ) + \delta_{12}(\Go \Md - \Fo\No) \nonumber \\
	\left[ x_{u_2}, \xi_{u_2}, \xi \right] =& \delta_{21} (\Fo\Md - \Eo\No) + \delta_{22}(\Go\Md - \Fo\No) \label{coefvv}.
	\end{align}
	Hence, from (\ref{coefuu}), (\ref{coefuv}) and (\ref{coefvv}) we can rewrite (\ref{eqdesenv}) as
	\begin{align*}
	\begin{pmatrix}
	u_{1}' & u_{2}'
	\end{pmatrix} \PP adj(\bm{\cal{II}}_{\Omega})\bm{\mathcal{I}}_{\Omega}\Delta_{\Omega}^{T}\begin{pmatrix}
	u_{1}'\\ u_{2}'
	\end{pmatrix} = 0.
	\end{align*} 
	
\end{proof}

\begin{teo}\label{teoprincipal}
Let $ \lbrace x, \xi \rbrace$ be a normal line congruence, where $\xi: U \rightarrow S^{2}$ is a proper frontal, $\Omega$ a tangent moving basis of $\xi$. Then the equation of principal surfaces is a multiple of the equation of developable surfaces by $\delta_{\Omega}$, where $\delta_{\Omega} = \det \Delta_{\Omega}$. More precisely
\begin{align}\label{desXprin}
\Delta_{\Omega} \PP adj(\bm{\cal{II}}_{\Omega})^{T}\Delta_{\Omega} \bm{\cal{I}}_{\Omega}\Delta_{\Omega}^{T} & = \delta_{\Omega} \PP adj(\bm{\cal{II}}_{\Omega}) \bm{\cal{I}}_{\Omega}\Delta_{\Omega}^{T}.
\end{align}
\end{teo}

\begin{proof}
	We know from propositions (\ref{obsprincipal}) and (\ref{propdesenv}) that the binary differential equations which provide principal and developable surfaces of the congruence are, respectively, given by
	\begin{align}
	\begin{pmatrix}
	u_{1}' & u_{2}'
	\end{pmatrix}\Delta_{\Omega} \PP adj(\bm{\cal{II}}_{\Omega})^{T}\Delta_{\Omega} \bm{\cal{I}}_{\Omega}\Delta_{\Omega}^{T}\begin{pmatrix}
	u_{1}' \\ u_{2}'
	\end{pmatrix} &= 0, \label{eqprincipal}\\
	\begin{pmatrix}
	u_{1}' & u_{2}'
	\end{pmatrix}\PP adj(\bm{\cal{II}}_{\Omega}) \bm{\cal{I}}_{\Omega}\Delta_{\Omega}\begin{pmatrix}
	u_{1}' \\ u_{2}'
	\end{pmatrix} &= 0.\label{eqdesenvovl}
	\end{align}
	It follows from proposition \ref{propcongnormal} that $\bm{\cal{II}}_{\Omega}adj(\Delta_{\Omega}^{T})$ is symmetric, since we have a normal congruence. It can be shown by straightforward calculations that  this is equivalent to say that	$adj(\bm{\cal{II}}_{\Omega})^{T}\Delta_{\Omega}$ is symmetric. Hence, we have
	\begin{align*}
	&\ \Delta_{\Omega} \PP adj(\bm{\cal{II}}_{\Omega})^{T}\Delta_{\Omega} \bm{\cal{I}}_{\Omega}\Delta_{\Omega}^{T}\\ 
	&= \Delta_{\Omega} \PP \Delta_{\Omega}^{T} adj(\bm{\cal{II}}_{\Omega}) \bm{\cal{I}}_{\Omega}\Delta_{\Omega}^{T} \nonumber\\
	&= -\Delta_{\Omega} \PP \Delta_{\Omega}^{T} \PP \PP adj(\bm{\cal{II}}_{\Omega}) \bm{\cal{I}}_{\Omega}\Delta_{\Omega}^{T},\; \text{since $-\PP = \PP^{-1}$}\nonumber\\
	&= \Delta_{\Omega}  adj(\Delta_{\Omega}) \PP adj(\bm{\cal{II}}_{\Omega}) \bm{\cal{I}}_{\Omega}\Delta_{\Omega}^{T},\; \text{since $-\PP\Delta_{\Omega}^{T} \PP = adj(\Delta_{\Omega})$} \nonumber\\
	&= \delta_{\Omega} \PP adj(\bm{\cal{II}}_{\Omega}) \bm{\cal{I}}_{\Omega}\Delta_{\Omega}^{T}.
	\end{align*}

\end{proof}

\begin{corollary}
Let $ \lbrace x, \xi \rbrace$ be a normal line congruence.	If $x$ and $\xi$ are analytic, then an analytic solution of the equation of principal surfaces is either a branch of $\Sigma(\xi)$ or an analytic solution of the equation of developable surfaces.
\end{corollary}
\begin{proof}
Let $\gamma: I \to U$, given by $\gamma(t) = (u_{1}(t), u_{2}(t))$, be an analytic solution of the equation of principal surfaces, then $\delta_{\Omega}(\gamma)$ is an analytic mapping. If there is $t_{0}$ such that the derivatives $\delta_{\Omega}(\gamma)^{(j)}(t_{0}) = 0$, for all positive integer $j$, then $\delta_{\Omega}(\gamma)(t) = 0$, for all $t \in I$. Otherwise, the zeros of $\delta_{\Omega}(\gamma)$ are isolated and therefore $\gamma$ is a solution of the equation of developable surfaces, once we have (\ref{desXprin}).
\end{proof}

From now on, we consider $x$ a frontal, $\Omega$ a tangent moving basis of $x$ and $\xi$ its normal vector field. Also note that we can take $\Omega$ as a tangent moving basis of $\xi$.
\begin{prop}\label{cor3}
Let $x: U \rightarrow \R^3$ and $\xi: U \rightarrow S^{2}$ be two proper frontals, such that $\xi$ is the unit normal vector field of $x$. Then, $w$ is principal direction, in the sense of \cite{Medina3}, if and only if this is a Kummer principal direction associated to the congruence $\lbrace x, \xi \rbrace$.
\end{prop}
\begin{proof}
It follows from \cite{Medina3} and from remark \ref{remarkindepenciabase} that the principal directions and the Kummer principal directions do not depend on the chosen tangent moving basis, so let us consider $\Omega$ an orthonormal one, i.e, $\Omega^{T}\Omega = id_{\R^2}$. From lemma 3.1 and remark 3.1 in \cite{Medina3}, it follows that $\bII_{\Omega}adj(\Lambda_{\Omega}^{T})$ is symmetric, then the principal directions are given by its eigenvectors. Since the congruence is normal, $\bm{\cal{II}}_{\Omega}adj(\Delta_{\Omega}^{T})$ is symmetric and we get analogously that the Kummer principal directions are given by its eigenvectors. As $D\xi = \Omega \Delta_{\Omega}^{T}$ and $Dx = \Omega \Lambda_{\Omega}^{T}$, then
\begin{align*}
\bII_{\Omega}adj(\Lambda_{\Omega}^{T}) &= -\Omega^{T}D\xi adj(\Lambda_{\Omega}^{T}) = -\Omega^{T}\Omega\Delta_{\Omega}^{T} adj(\Lambda_{\Omega}^{T}) =- \Delta_{\Omega}^{T}adj(\Lambda_{\Omega}^{T})\\
\bm{\cal{II}}_{\Omega}adj(\Delta_{\Omega}^{T}) &= -\Omega Dx = -\Omega^{T}\Omega\Lambda_{\Omega}^{T}adj(\Delta_{\Omega}^{T}) = - \Lambda_{\Omega}^{T}adj(\Delta_{\Omega}^{T}).\end{align*}
Hence, $\bII_{\Omega}adj(\Lambda_{\Omega}^{T}) = adj(\bm{\cal{II}}_{\Omega}adj(\Delta_{\Omega}^{T}))$. Then, the eigenvectors of $\bII_{\Omega}adj(\Lambda_{\Omega}^{T})$ are the eigenvectors of $\bm{\cal{II}}_{\Omega}adj(\Delta_{\Omega}^{T})$, that is, $w \in P_{\Omega}$ is a principal direction if and only if is a Kummer principal direction.
\end{proof}

\begin{obs}\label{obscor3}\normalfont
	It follows from $\bII_{\Omega}adj(\Lambda_{\Omega}^{T}) = adj(\bm{\cal{II}}_{\Omega}adj(\Delta_{\Omega}^{T}))$ that if $w_{1}, w_{2}$ are the eigenvectors of $\bm{\cal{II}}_{\Omega}adj(\Delta_{\Omega}^{T})$ associated to the eigenvalues $\gamma_{1}, \gamma_{2}$, respectively, then $w_{1}, w_{2}$ are the eigenvectors of $\bII_{\Omega}adj(\Lambda_{\Omega}^{T})$ associated to the eigenvalues $\gamma_{2}, \gamma_{1}$, respectively. Furthermore, if $w_{1}$, for instance, is a Kummer principal direction associated to a maximum of $\mathscr{K}_{q}^{\Omega}$, then $w_{1}$ is a principal direction associated to a minimum of the $\Omega$-relative normal curvature, which is defined in \cite{Medina3}. That is, if we denote by $\mathscr{K}_{1\Omega}$ and $\mathscr{K}_{2\Omega}$ the minimum and the maximum of $\mathscr{K}_{q}^{\Omega}$, respectively, then $k_{1\Omega} = \mathscr{K}_{2\Omega}$ and $k_{2\Omega} = \mathscr{K}_{1\Omega}$ where $k_{1\Omega}$ and $k_{2\Omega}$ are the $\Omega$-relative principal curvatures (see section \ref{sec:back3}).
\end{obs}

\begin{ex}\normalfont
Despite proposition \ref{cor3}, it is not true that for a line congruence given by a frontal $x: U \to \R^3$ and its unit normal vector field $\xi: U \to S^2$, a curve is a Kummer principal line if and only if it is a line of curvature. Let us take, for instance, the congruence given by
\begin{align*}
x(u_{1}, u_{2}) &= \left(u_{1}, u_{2}, u_1^2u_2+u_2^2\right)\\
\xi(u_{1}, u_{2}) &= \dfrac{1}{\sqrt {4u_1^2u_2^2+u_1^4+4u_1^2u_2+4u_2^2+1}}\left(2u_1u_2, -u_1^2-2u_{2}, 1\right).
\end{align*}
In this case, the Gaussian curvature is given by 
\begin{align*}
K(u_{1}, u_{2}) = {\dfrac {-4({u_{1}}^{2}-u_{2})}{ \left( {u_{1}}^{4}+4{u_{1}}^{2}{u_{2
		}}^{2}+4{u_{1}}^{2}u_{2}+4{u_{2}}^{2}+1 \right)^{2}}}
\end{align*}
 so $u_{2} = u_{1}^2$ is a curve of parabolic points. The equation of the lines of curvature is given by
\begin{align*}
&(2u_{1}^{3}u_{2}^{2}-4u_{1}u_{2}^{3}+u_{1})u_{1}'^{2}+
(-u_{1}^{4}u_{2}-4u_{2}^{3}-u_{2}+1)u_{2}'u_{1}'\\
&+ (-
	u_{1}^{5}-2u_{1}^{3}u_{2}-u_{1})u_{2}'^{2}=0.
\end{align*}
On the other hand, the equation of principal surfaces of the congruence is given by
\begin{align*}
&\left( u_{2}-u_{1}^{2} \right)\left[  \left( 2u_{1}^{3}u_{2}^{2}-4u_{1}u_{2}^{3}+u_{1}
\right)u_{1}'^{2} + \left( -u_{1}^{4}u_{2}-4
u_{2}^{3}-u_{2}+1 \right)u_{2}'u_{1}' \right. \\  & \left. +\left( -u_{1}^{5}-2u_{1}^{3}u_{2}-u_{1} \right)u_{2}'^{2} \right] = 0.
\end{align*}
Then, the curve of parabolic points $ u_{2} = u_{1}^{2}$ is a Kummer principal line, but it is not a line of curvature.
\end{ex}

Next, we have some results regarding the relation between the Kummer principal lines and the lines of curvature of a frontal.
\begin{corollary}\label{corol43}
	Let $x: U \rightarrow \R^3$ be a frontal, $\Omega$ a tangent moving basis of $x$ and $\xi: U \rightarrow S^{2}$ a normal vector field induced by $\Omega$. If $\xi$ is a proper frontal, then the equation of principal surfaces is
	\begin{align}\label{cor2}
	\gamma'^{T}\Delta_{\Omega} \PP adj(\bm{\cal{II}}_{\Omega})^{T}\Delta_{\Omega} \bm{\cal{I}}_{\Omega}\Delta_{\Omega}^{T}\gamma' &= -K_{\Omega}\det(\mathbf{I}_{\Omega})\gamma'^{T}\PP\alpha^{T}_{\Omega}(\gamma)\gamma',
	\end{align}
	where $K_{\Omega}$ and $\alpha_{\Omega}$ are related with $x$ (see section \ref{sec:back3}). 
\end{corollary}
\begin{proof}
	Note that  $\Delta_{\Omega} =\mu_{\Omega}$, $K_{\Omega}=\det{\Delta_{\Omega}}$, $\bm{\cal{I}}_{\Omega}=\mathbf{I}_{\Omega}$ and $\bm{\cal{II}}_{\Omega}=-\mathbf{I}_{\Omega}\Lambda_{\Omega}^T$. Thus, from (\ref{desXprin})
	\begin{align*}
	\Delta_{\Omega} \PP adj(\bm{\cal{II}}_{\Omega})^{T}\Delta_{\Omega} \bm{\cal{I}}_{\Omega}\Delta_{\Omega}^{T} &=\delta_{\Omega} \PP adj(\bm{\cal{II}}_{\Omega}) \bm{\cal{I}}_{\Omega}\Delta_{\Omega}^{T}\\
	&=K_{\Omega} \PP adj(-\mathbf{I}_{\Omega}\Lambda_{\Omega}^T) \mathbf{I}_{\Omega}\mu_{\Omega}^{T}\\
	&=K_{\Omega} \PP adj(\Lambda_{\Omega}^T)adj(-\mathbf{I}_{\Omega}) \mathbf{I}_{\Omega}\mu_{\Omega}^{T}\\
	&=-K_{\Omega}\det(\mathbf{I}_{\Omega})\PP\alpha^{T}_{\Omega}(\gamma),
	\end{align*}
	where $Dx=\Omega\Lambda_{\Omega}^T$.
\end{proof}

Note that via corollary \ref{corol43} we can express the equation of principal surfaces of the congruence (\ref{repmatricialpri}) only considering quantities related to the frontal $x$, when we have an exact normal congruence.

\begin{obs}\normalfont
	It is worth observing above that $-\det(\mathbf{I}_{\Omega})\gamma'^{T}\PP\alpha^{T}_{\Omega}(\gamma)\gamma'=0$ is the equation of the developable surfaces of the congruence.
\end{obs}

\begin{corollary}\label{cor33}
	Let $x: U \rightarrow \R^3$ and $\xi: U \rightarrow S^{2}$ be two proper frontals with the same singular sets, such that $\xi$ is the unit normal vector field of $x$. Then, a curve on $x$ is a Kummer principal line if and only if it is a line of curvature of $x$.
\end{corollary}
\begin{proof}
	The results follows from the fact that	$\lambda_{\Omega}\gamma'^{T}\PP\alpha^{T}_{\Omega}(\gamma)\gamma'=0$ is the equation of the lines of curvature (see corollary 5.1 in \cite{Medina3}) and from $K_{\Omega}^{-1}(0)=\lambda_{\Omega}^{-1}(0)$.
\end{proof}

\begin{corollary}\label{cor44}
	Let $x: U \rightarrow \R^3$ be a proper frontal with extendable normal curvature, such that the extension of the Gaussian curvature $K$ never vanishes. Then, the Kummer principal lines coincide with the lines of curvature of $x$.
\end{corollary}
\begin{proof}
	It follows from corollary 3.1 in \cite{Medina3} that $x$ and its normal vector field have the same singular set, hence applying corollary \ref{cor3} we have the result.
\end{proof}

\begin{ex}\normalfont
	Let $x: U \rightarrow \R^3$ defined by $x = \left(u_{1},  \frac{2}{5}u_{2}^5 + u_{2}^2, u_{1}u_{2}^2  \right)$, for $ U = \left(-1,1 \right) \times (-1, 1)$ (Figure \ref{frontalnp}). Then $Dx = \Omega \Lambda_{\Omega}^{T}$, where
	\begin{align}
	\Omega = \begin{pmatrix}
	1 & 0\\
	0 & u_{2}^3 + 1\\
	u_{2}^2 & u_{1}
	\end{pmatrix}\; \text{and}\; \Lambda_{\Omega} = \begin{pmatrix}
	1 & 0\\
	0 & 2u_{2}
	\end{pmatrix}.
	\end{align}
	The unit normal vector field induced by $\Omega$ is given by 
	\begin{align*}
\xi = \frac{1}{\mu
}(-{u_{2}}^{2}( u_{2}+1)({u_{2}}^{2}-u_{2}+1), -u_{1}, ( u_{2}+1)  ( {u_{2}}^{2}-u_{2}+1) ),	
	\end{align*}
	where $\mu = \sqrt{{u_{2}}^{10}+2{u_{2}}^{7}+{u_{2}}^{6}+{u_{2}}^{4}+2{u_{2}}^{3}+{u_
			{1}}^{2}+1}$.
The frontal $x$ in this example is special, because it is a frontal with extendable normal curvature without false singularities (see comments after theorem 3.2 in \cite{Medina3}). Furthermore, $\lambda_{\Omega} = 2u_{2}$ and 
\begin{align*}
K_{\Omega} = \dfrac{2u_{2}(u_{2} + 1)^2(u_{2}^2 - u_{2} + 1)^2}{(u_{2}^{10} + 2u_{2}^7 + u_{2}^6 + u_{2}^4 + 2u_{2}^3 + u_{1}^2 + 1)^2},\end{align*}
therefore, considering that at a regular point the Gaussian curvature is given by $\dfrac{K_{\Omega}}{\lambda_{\Omega}}$, we obtain that
\begin{align*}
K = \dfrac{(u_{2} + 1)^2(u_{2}^2 - u_{2} + 1)^2}{(u_{2}^{10} + 2u_{2}^7 + u_{2}^6 + u_{2}^4 + 2u_{2}^3 + u_{1}^2 + 1)^2}
\end{align*}
 is the extension of the Gaussian curvature to $U$. Then, the Gaussian curvature also admits an extension to $U$ and in this case, the extension is non-vanishing. By applying corollary \ref{cor44}, the Kummer principal lines coincide with the lines of curvature of $x$, which are given by the implicit differential equation
 \begin{align*}
 &2u_{2}\left[\left( {u_{2}}^{7}+{u_{2}}^{4}+{u_{2}}^{3}+1 \right) {u_{1}'}^{2}+ \left( 3u_{1}{u_{2}}^{6}+3u_{1}{u_{
 		2}}^{2} \right) u_{1}'u_{2}' \right]\\
 &+ 2u_2\left[\left( -4{u_{2}}^{11}-12{u_{2}}^{8}+2{u_{1}}^{2}{u_{2}
 }^{5}-12{u_{2}}^{5}-4{u_{1}}^{2}{u_{2}}^{2}-4{u_{2}}^{2}
 \right) u_{2}'^{2}\right] = 0.
 \end{align*}
\begin{figure}[ht!]
 	\includegraphics[scale=0.27]{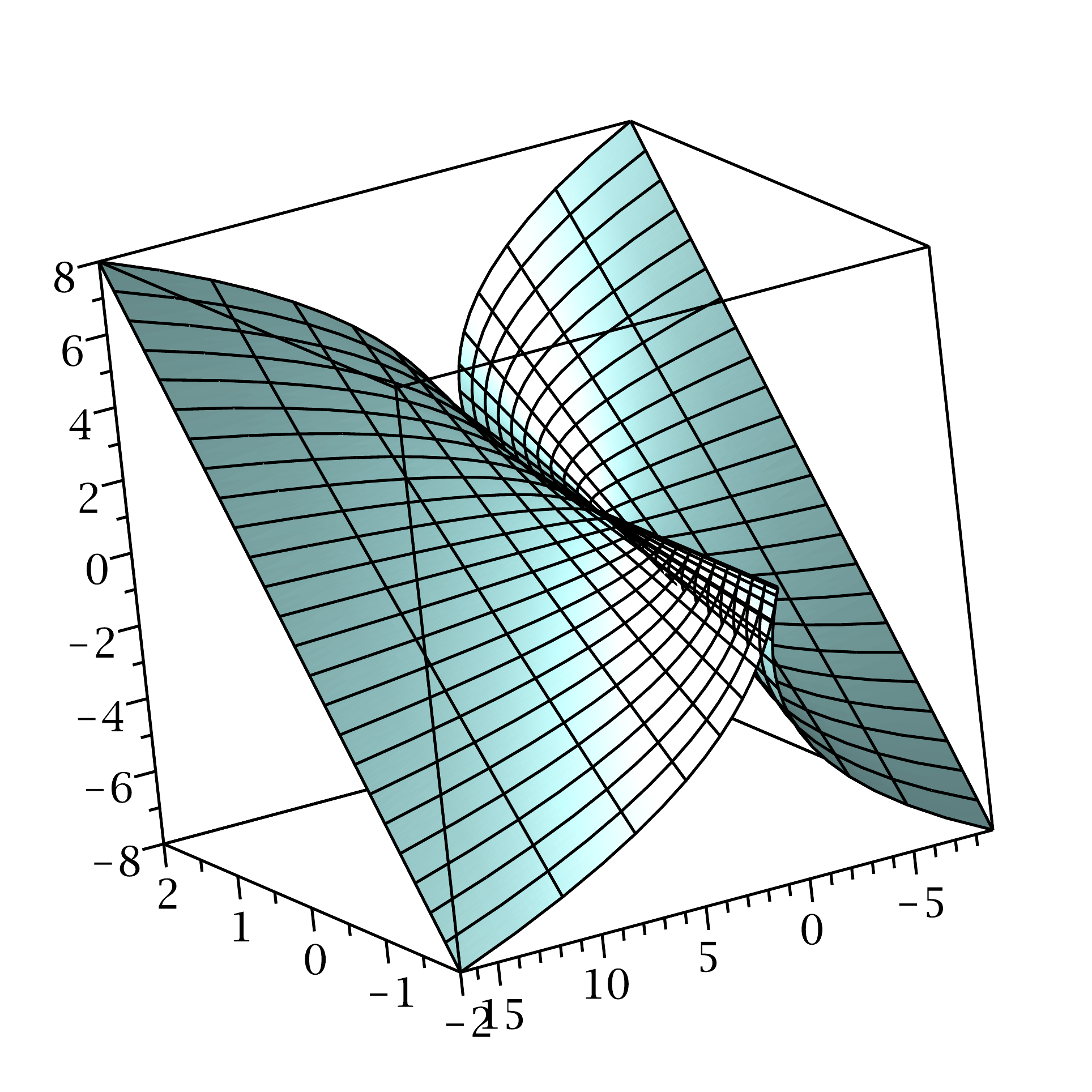}
 	\includegraphics[scale=0.27]{lambdaeq1-eps-converted-to.pdf}
 	\caption{Frontal for which the unit normal vector has the same singular set.}
 	\label{frontalnp}
 \end{figure}
\end{ex}

{\footnotesize \bibliographystyle{siam}
	\bibliography{bibi}}

\end{document}